\documentclass[12pt]{article}

\usepackage{fullpage}
\usepackage{amsfonts,amssymb,stmaryrd,amsmath}
\usepackage{graphicx,tikz,tabularx}
\usepackage{amsthm,mathtools}

\newcommand\longto{{\longrightarrow}}

\newcommand\nr{\mathrm{nr}}

\newcommand\CC{\mathbb{C}}\newcommand\QQ{\mathbb{Q}}
\newcommand\NN{\mathbb{N}}\newcommand\ZZ{\mathbb{Z}}

\newcommand\Nb{{\operatorname{Nb}}}
\newcommand\Horn{{\operatorname{Horn}}}
\newcommand\GL{{\operatorname{GL}}}\newcommand\SL{{\operatorname{SL}}}

\usepackage[utf8]{inputenc}  

\usepackage[T1]{fontenc}

\usepackage{hyperref}
\usepackage{url}

\newtheorem{prop}{Proposition}
\newtheorem{conj}{Conjecture}
\newtheorem{theo}[prop]{Theorem}

\newtheorem{coro}[prop]{Corollary}

\newenvironment{rem}{\noindent{\bf Remark.}}{}

\newcommand\DynkinNodeSize{2mm}
\newcommand\DynkinArrowLength{3mm}
\tikzset{
  dnode/.style={
    circle,
    inner sep=0pt,
    minimum size=\DynkinNodeSize,
    fill=white,
    draw},
  middlearrow/.style={
    decoration={markings,
      mark=at position 0.6 with
      {\draw (0:0mm) -- +(+135:\DynkinArrowLength); \draw (0:0mm) -- +(-135:\DynkinArrowLength);},
    },
    postaction={decorate}
  },
  leftrightarrow/.style={
    decoration={markings,
      mark=at position 0.999 with
      {
      \draw (0:0mm) -- +(+135:\DynkinArrowLength); \draw (0:0mm) -- +(-135:\DynkinArrowLength);
      },
      mark=at position 0.001 with
      {
      \draw (0:0mm) -- +(+45:\DynkinArrowLength); \draw (0:0mm) -- +(-45:\DynkinArrowLength);
      },
    },
    postaction={decorate}
  },
  sedge/.style={
  },
  dedge/.style={
    middlearrow,
    double distance=0.5mm,
  },
  tedge/.style={
    middlearrow,
    double distance=1.0mm+\pgflinewidth,
    postaction={draw}, 
  },
  infedge/.style={
    leftrightarrow,
    double distance=0.5mm,
  }
}

\begin{document}
\title{Some unexpected properties of Littlewood-Richardson coefficients}
\author{Maxime Pelletier and Nicolas Ressayre}

\maketitle

\begin{abstract}
We are interested in identities between Littlewood-Richardson
coefficients, and hence in comparing different tensor product decompositions of the irreducible
modules of the linear group ${\operatorname{GL}}_n({\mathbb C})$. 
A family of partitions -- called near-rectangular -- is defined, and we prove a stability result which basically asserts that the
decomposition of the tensor product of two representations associated to near-rectangular
partitions does not depend on $n$. 
Given a partition $\lambda$, of length at most $n$, denote by
$V_n(\lambda)$ the associated simple ${\operatorname{GL}}_n({\mathbb
  C})$-module.
We conjecture that, if $\lambda$ is near-rectangular and $\mu$ any
partition, the decompositions of $V_n(\lambda)\otimes V_n(\mu)$
and $V_n(\lambda)^*\otimes V_n(\mu)$
coincide modulo a mysterious bijection. We prove this conjecture if
$\mu$ is also near-rectangular and report several computer-assisted
computations which reinforce our conjecture.
\end{abstract}

\section{Introduction}

In this paper we study some properties of Littlewood-Richardson
coefficients when some partitions are near-rectangular (see below for
a precise definition).
Let $n\geq 2$ be an integer. The irreducible
representations of $\GL_n(\CC)$ are parametrized by all non-increasing sequences of $n$ integers. As is often the case, we focus on the polynomial representations among those, which correspond to the sequences containing only non-negative integers, also called partitions with at most $n$ parts.  
Denote by $V_n(\lambda)$ the
representation of $\GL_n(\CC)$ associated to such a partition $\lambda$. The Littlewood-Richardson
coefficient $c_{\lambda\mu}^\nu $ appears in the tensor product decomposition
$$
V_n(\lambda)\otimes V_n(\mu)=\bigoplus_{\nu\in\Lambda_n}c_{\lambda\mu}^\nu V_n(\nu),
$$
where $\Lambda_n$ denotes the set of partitions of length at  most
$n$. Denote moreover by $V_n(\lambda)^*$ the $\GL_n(\CC)$-module which
is dual to $V_n(\lambda)$.

Obviously, given $\lambda$ and $\mu$ two partitions with at most $n$ parts, $V_n(\lambda)\otimes V_n(\mu)$ and $V_n(\lambda)^*\otimes
V_n(\mu)$ are not isomorphic as $\GL_n(\CC)$-modules. 
Nevertheless, we may want to compare their decompositions in irreducible
modules.\\

\noindent{\bf Problem 1.} Compare the two $\GL_n(\CC)$-modules,
$V_n(\lambda)\otimes V_n(\mu)$ and $V_n(\lambda)^*\otimes V_n(\mu)$.\\

For example, Coquereaux-Zuber \cite{CZ:sumtensor} showed that the sums of the
multiplicities of these two representations coincide.
For $\lambda\in \Lambda_n$, set $\lambda^*=\lambda_1-\lambda_n\geq
\lambda_1-\lambda_{n-1}\geq\cdots\geq\lambda_1-\lambda_2\geq 0$. 
Then $V_n(\lambda)^*$ and $V_n(\lambda^*)$ only differ by a tensor
power of the determinant. More precisely, $V_n(\lambda)^*$ is the irreducible representation of $\GL_n(\CC)$ corresponding to the sequence $-\lambda_n\geq -\lambda_{n-1}\geq\dots\geq -\lambda_1$. As a consequence $V_n(\lambda^*)$ is isomorphic to $V_n(\lambda)^*$ tensored by the $\lambda_1$-th power of the determinant representation of $\GL_n(\CC)$.

Following these observations, the Coquereaux-Zuber's result can be written
as 
\begin{equation}
  \label{eq:3}
  \sum_{\nu\in\Lambda_n}c_{\lambda\mu}^\nu=\sum_{\nu\in\Lambda_n}c_{\lambda^*\mu}^\nu.
\end{equation}
More generally, we may want to compare
$\{c_{\lambda\mu}^\nu\,:\,\nu\in \Lambda_n\}$ and
$\{c_{\lambda^*\mu}^\nu\,:\,\nu\in \Lambda_n\}$ as multisets. For instance, for $n=3$, $\lambda = (5, 3)$ and $\mu = (6, 3)$, $V_3(\lambda) \otimes V_3(\mu)$ decomposes as
$$
\begin{array}{l}
V_3(7,5,5) + V_3(7,7,3) + V_3(8,8,1) + V_3(9,4,4) + V_3(9,8) + V_3(10,7) \\+ V_3(11,6)
+V_3(11,3,3) + V_3(11,4,2) + V_3(11,5,1) + V_3(6,6,5)\\
+2V_3(7,6,4) + 2V_3(8,5,4) + 2V_3(8,7,2) + 2V_3(9,7,1) + 2V_3(10,4,3)\\
  + 2V_3(10,5,2) 
+ 2V_3(10,6,1) +3V_3(8,6,3) + 3V_3(9,5,3) + 3V_3(9,6,2),
\end{array}
$$
while $V_3(\lambda^*) \otimes V_3(\mu)$ expands as
$$
\begin{array}{l}
V_3(7,7,2) + V_3(8,4,4) + V_3(10,3,3) + V_3(8,8) + V_3(9,7) + V_3(10,6) \\
+V_3(11,3,2) + V_3(11,4,1) + V_3(6,5,5) + V_3(6,6,4)+2V_3(7,5,4)\\ 
+ 2V_3(7,6,3) + 2V_3(8,7,1) + 2V_3(9,4,3) + 2V_3(9,6,1) + 2V_3(10,4,2)\\
+ 2V_3(10,5,1)+3V_3(8,5,3) + 3V_3(8,6,2) + 3V_3(9,5,2) + V_3(11,5).
\end{array}
$$
One can then notice that the multiplicities in the two expansions are the same: 11 occurrences
of ``1'', 7 occurrences of ``2'' and 3 occurrences of ``3'' in both
cases. A natural question is then: is this always true?

We check in this article that it is the case for $\GL_3(\CC)$ in general (note that this was already proven by Coquereaux and Zuber in \cite{CZ:GL3}; they even give several proofs of this fact in this paper from 2014): using a computer, we are able to compute explicitly, for $n=3$, the function
$$
\begin{array}{ccl}
(\Lambda_n^2)\times\NN&\longto&\NN\\
(\lambda,\mu,c)&\longmapsto& \Nb_n(c_{\lambda\mu}^\bullet> c):=
\#\{
\nu\in\Lambda_n\,:\, c_{\lambda\mu}^\nu> c
\}.
\end{array}
$$
See Section~\ref{sec:SL3} for details.
In this introduction we report on this computation as follows:

\begin{prop}
\label{prop:Nc3}
  The function 
$$
\begin{array}{cccl}
  \Nb_3(c_{\lambda\mu}^\bullet> c)\,:&\Lambda_3\times \Lambda_3\times\NN&\longto &\NN\\
&(\lambda,\mu,c)&\longmapsto&\#\{
\nu\in\Lambda_n\,:\, c_{\lambda\mu}^\nu>c
\}
\end{array}
$$
is piecewise polynomial of degree 2 with 7 cones\footnote{\label{foot1}Namely, the
  cone $\Lambda_3\times \Lambda_3\times\NN$ is the support of a fan
  with 7 maximal cones. And,  on each one of these 7 cones, the
  function $\Nb_3(c_{\lambda\mu}^\bullet> c)$ is given by a polynomial
  in $(\lambda,\mu,c)$ of degree 2.}. Moreover,
\begin{equation}
  \label{eq:2}
  \Nb_3(c_{\lambda\mu}^\bullet> c)=\Nb_3(c_{\lambda^*\mu}^\bullet> c).
\end{equation}
\end{prop}

\bigskip
We then aim to suggest a generalization to $\GL_n(\CC)$ for any $n$: we will say that a partition $\lambda\in\Lambda_n$ is {\it
  near-rectangular} if $\lambda=\lambda_1\lambda_2^{n-2}\lambda_n$ for some integers
$\lambda_1\geq \lambda_2\geq \lambda_n$; that is, if $\lambda_2=\cdots=\lambda_{n-1}$. 
We formulate the following
\begin{conj}
\label{conj:setLR}
Let $\lambda$ and $\mu$ in $\Lambda_n$.
  If $\lambda$ is near-rectangular then
$$\forall c\in\NN\qquad
\#\{\nu\in \Lambda_n \,:\, c_{\lambda\mu}^\nu=c\}=
\# \{\nu\in \Lambda_n \,:\, c_{\lambda^*\mu}^\nu=c\}.
$$
\end{conj}

\bigskip
Equivalently, we conjecture that there exists a bijection
$\varphi\,:\,\Lambda_n\longto\Lambda_n$, depending on $\lambda$ and
$\mu$ such that 
$$
\forall \nu\in \Lambda_n\qquad c_{\lambda\mu}^\nu=c_{\lambda^*\mu}^{\varphi(\nu)},
$$
if $\lambda$ is near-rectangular.
In other words, the multisets $\{c_{\lambda\mu}^\nu\,:\,
\nu\in\Lambda_n\}$ and $\{c_{\lambda^*\mu}^\nu\,:\,
\nu\in\Lambda_n\}$ are expected to be equal, if $\lambda$ is near-rectangular.

Note that any partition of size $\leq 3$ is near-rectangular. Then,
both the last assertion of Proposition~\ref{prop:Nc3}, and the result of \cite{CZ:GL3} are equivalent to
\begin{coro}
\label{cor:GL3}
  Conjecture~\ref{conj:setLR} holds for $n=3$.
\end{coro}

\bigskip
In literature, there are a lot of equalities between
Littlewood-Richardson coefficients. 
There are symmetries (see \cite{bri144} and references therein), stabilities (see
\cite{Briand_2015}), reductions (see \cite{SD:reductionLR}). None of these numerous
results really seems to explain whether or why Conjecture~\ref{conj:setLR} should hold. 
Moreover, there are various combinatorial models for the
Littlewood-Richardson coefficients (see \cite{Ful:YT,Litt:path,Zelevinsky:pictures,KT:saturation,Vakil:LR,Coskun:2step}). 
None of these models seems to prove Conjecture~\ref{conj:setLR} easily.

What made this Conjecture~\ref{conj:setLR} even more unexpected for us is that it seems that, should the aforementioned bijection $\varphi\,:\,\Lambda_n\longto\Lambda_n$ exist, it does not seem to be expressible simply in terms of $\lambda$ and $\mu$. Indeed, even if $n=3$, we remark in Section~\ref{sec:SL3} that $(\lambda,\mu,\nu)\longto(\lambda^*,\mu,\varphi(\nu))$ cannot be linear.\\
 
We checked, using Sagemath, Conjecture~\ref{conj:setLR} on a few million examples for
$\GL_4(\CC)$, $\GL_5(\CC)$, $\GL_6(\CC)$, and $\GL_{10}(\CC)$: see Section~\ref{sec:typeAn} for details. Moreover, we also give in Section~\ref{sec:typeAn} an example showing that the assumption
on $\lambda$ is truly necessary.

\bigskip
We now consider the tensor products of representations associated to two near-rectangular partitions.
Observe that a partition $\lambda$ of length $l$ parametrizes representations
$V_n(\lambda)$ of $\GL_n(\CC)$ for any $n\geq l$. 
A priori, $c_{\lambda\mu}^\nu $ could depend on $n$. It is a classical
result (see e.g. \cite{Ful:YT}) of stability that  it actually does not.
Our second result is a similar stability result but for partitions of
arbitrarily large length. 
Indeed, fix two near-rectangular partitions   $\lambda=\lambda_1\,
\lambda_2^{n-2}\lambda_n$ and $\mu=\mu_1\, \mu_2^{n-2}\mu_n$. 
We prove that the decomposition of $V_n(\lambda)\otimes
V_n(\mu)$ does not depend on $n\geq 4$ but only on the six
integers $\lambda_1,\lambda_2,\lambda_n,\mu_1,\mu_2$ and $\mu_n$ . More precisely,
we get a simple expression of the Littlewood-Richardson coefficients that
appear in this tensor product, expression independent of $n$.

\begin{prop}
\label{prop:stab}
  Let $n\geq 4$. Let $\lambda=\lambda_1\,
\lambda_2^{n-2}$ and $\mu=\mu_1\, \mu_2^{n-2}$ be two near-rectangular partitions.
Let  $\nu$ be a partition with at most $n$ parts.
Then $c_{\lambda\mu}^\nu =0$ unless $\nu=\nu_1\,\nu_2\,(\lambda_2+\mu_2)^{n-4}\,\nu_{n-1}\nu_n$
for four integers $\nu_1,\,\nu_2,\,\nu_{n-1}$ and $\nu_n$ such that $\nu_1\geq \nu_2\geq
\lambda_2+\mu_2\geq \nu_{n-1}\geq \nu_n$. In this case
\[ 
  c_{\lambda\mu}^\nu=\#\big(\llbracket M,m\rrbracket\big),
 \]
where 
$$
M=\max(0,\lambda_2+\mu_1-\nu_1,-\mu_2+\nu_n)
$$
and
$$
m=\min(\lambda_1+\mu_1-\nu_1,\lambda_2+\mu_1-\nu_2,-\lambda_2-\mu_2+\nu_{n-1}+\nu_n,-\mu_2+\nu_{n-1}).
$$
In particular, this value does not depend on $n\geq 4$.
\end{prop}

\bigskip
In Proposition~\ref{prop:stab}, for simplicity we assumed that
$\lambda_n=\mu_n=0$. Yet nothing is lost with such an assumption, since 
$V_n(\lambda_1\lambda_2^{n-2}\lambda_n)\simeq\det^{\lambda_n}\otimes
V_n((\lambda_1-\lambda_n) (\lambda_2-\lambda_n)^{n-2})$.

Proposition~\ref{prop:stab} positively answers
\cite[Question~2]{PW:multiplicities} by giving a much stronger result.
 It also allows to check a
particular case of Conjecture~\ref{conj:setLR}. 
Indeed,
with the help of a computer we computed
$\Nb_4(c_{\lambda\mu}^\bullet> c)$ for $\lambda$ and  $\mu$
near-rectangular.
Let $\Lambda_n^\nr=\{\lambda_1\lambda_2^{n-2}\lambda_n\,:\,\lambda_1\geq \lambda_2\geq \lambda_n\}$  be the set of near-rectangular partitions of
length at most $n$.

\begin{prop}\label{prop:introSL4nr2}
  The function 
$$
\begin{array}{cccl}
\Nb_4(c_{\lambda\mu}^\bullet> c)\,:&  \Lambda_4^\nr\times \Lambda_4^\nr\times\NN&\longto &\NN\\
&(\lambda,\mu,c)&\longmapsto&\#\{
\nu\in\Lambda_n\,:\, c_{\lambda\mu}^\nu >c
\}
\end{array}
$$
is piecewise polynomial of degree 3 with 36 cones. Moreover,
\begin{equation}
 \Nb_4(c_{\lambda\mu}^\bullet> c)=\Nb_4(c_{\lambda^*\mu}^\bullet> c).
\end{equation}
\end{prop}

The 36 polynomial functions and cones are given in
Section~\ref{sec:SL4nr2}.
 As a consequence of Propositions~\ref{prop:stab} and \ref{prop:introSL4nr2}, we get

\begin{coro} Let $n\geq 4$. 
  Conjecture~\ref{conj:setLR} holds for $\GL_n(\CC)$ assuming in addition
  that $\mu$ is near-rectangular.
\end{coro}

A much weaker version of Conjecture~\ref{conj:setLR} is

  \begin{conj}
\label{conj:nbLR}
  If $\lambda\in\Lambda_n$ is near-rectangular then
$$
\# \{\nu\in \Lambda_n \,:\,c_{\lambda\mu}^\nu\neq 0\}=
\# \{\nu\in \Lambda_n \,:\,c_{\lambda^*\mu}^\nu\neq 0\}.
$$
\end{conj}

Equivalently, we ask whether, for $\lambda\in\Lambda_n^\nr$, 
$$
\forall\mu\in\Lambda_n\qquad \Nb_n(c_{\lambda\mu}^\bullet>0)=\Nb_n(c_{\lambda^*\mu}^\bullet> 0).
$$
For $n=4$ and $\lambda$ near-rectangular,
 we computed $\Nb_4(c_{\lambda\mu}^\bullet>0)$ and checked
 Conjecture~\ref{conj:nbLR}. Here we report on this computation as follows (see Section~\ref{sec:GL4conj2} for details).

\begin{prop}
  The function 
$$
\begin{array}{cccl}
\Nb_4(c_{\lambda\mu}^\bullet> 0)\,:&  \Lambda_4^\nr\times \Lambda_4&\longto &\NN\\
&(\lambda,\mu)&\longmapsto&\#\{
\nu\in\Lambda_n\,:\, c_{\lambda\mu}^\nu> 0
\}
\end{array}
$$
is piecewise quasi-polynomial\footnote{See Section~\ref{sec:Horncone}
  right after Proposition~\ref{prop:nbisquasipol} for a definition.} of degree 3 with 205 cones.
The only congruence occurring is the parity of $\lambda_1+|\mu|$.
Moreover,
\begin{equation}
 \Nb_4(c_{\lambda\mu}^\bullet> 0)=\Nb_4(c_{\lambda^*\mu}^\bullet> 0).\label{eq:4nb}
\end{equation}
This symmetry with the complete duality $(\lambda,\mu)\longmapsto
(\lambda^*,\mu^*)$ gives an action of $(\ZZ/2\ZZ)^2$. This group acts
on the 205 pairs (cone,quasi-polynomial) with 83 orbits.
\end{prop}

\bigskip
This work is based on numerous computer aided computations with Barvinok
\cite{iscc}, Normaliz \cite{normaliz} and SageMath \cite{sage}. Details on these computations can be found on
the webpage of the second author \cite[Supplementary material]{wp}.

\bigskip
\noindent{\bf Acknowledgements.} We are very grateful to Vincent Loechner, who helped us in the use
of ISCC implementation of Barvinok's algorithm.  
We also want to thank Dipendra Prasad for useful discussions on
\cite{PW:multiplicities}, which motivated this work.
Finally, we are very grateful to the anonymous referee for his
numerous advice which helped improving the article.

The authors are partially supported by the French National Agency
(Project GeoLie ANR-15-CE40-0012).

\bigskip
\begin{rem}
After a version of this work was posted on ArXiv, Darij Grinberg
offered a solution of our main conjecture in
\cite{DG}. Therein he defines a piecewise linear involution $\varphi$  from
$\ZZ^n$ to $\ZZ^n$ satisfying
$$
\forall \nu\in \Lambda_n\qquad c_{\lambda\mu}^\nu=c_{\lambda^*\mu}^{\varphi(\nu)},
$$
if $\lambda$ is near-rectangular, thus solving our conjecture. 
An amazing fact is that
this bijection does not necessarily map a partition to a partition:
if $\varphi(\nu)$ is not a partition then  $c_{\lambda\mu}^\nu$ simply
vanishes, allowing $\varphi$ to work.
\end{rem}

\section{Generalities on the function $\Nb_n(c_{\lambda\mu}^\bullet> c)$}
\label{sec:Horncone}

Recall that for $\lambda,\mu\in\Lambda_n$ and $c\in\NN$ we set
$$
\Nb_n(c_{\lambda\mu}^\bullet>c)=\#\{\nu\in\Lambda_n\,:\, c_{\lambda\mu}^\nu>c\}.
$$

Since $V_n(\lambda)\otimes V_n(\mu)\simeq V_n(\mu)\otimes
V_n(\lambda)\simeq (V_n(\lambda^*)\otimes V_n(\mu^*))^*$ as $\SL_n(\CC)$-modules,  the
function $\Nb_n(c_{\lambda\mu}^\bullet>c)$ satisfies
\begin{equation}
  \label{eq:symNb}
  \Nb_n(c_{\lambda\mu}^\bullet>c)=\Nb_n(c_{\mu\lambda}^\bullet>c)=\Nb_n(c_{\lambda^*\mu^*}^\bullet>c)=\Nb_n(c_{\mu^*\lambda^*}^\bullet>c).
\end{equation}

Let $1^n\in\Lambda_n$ denote the partition  with $n$ parts equal to
$1$. Then $V_n(1^n)$ is the one dimensional representation of
$\GL_n(\CC)$ given by the determinant. In particular, as an
$\SL_n(\CC)$-module $V_n(\lambda)\simeq V_n(\lambda-k^n)$ 
for any $k$.
Set $\Lambda_n^0=\{\lambda\in\Lambda_n\,:\,\lambda_n=0\}$.
It can be seen as the set of dominant weights for the group  $\SL_n(\CC)$.
For $\lambda\in\Lambda_n$, set $\bar
\lambda=\lambda-\lambda_n^n$, the partition obtained by substracting
$\lambda_n$ to each part of $\lambda$. Then

\begin{equation}
  \label{eq:NbredSL}
  \Nb_n(c_{\lambda\mu}^\bullet>c)=\Nb_n(c_{\bar\lambda\bar\mu}^\bullet>c), 
\qquad\mathrm{and\ even\ more}\qquad
c_{\lambda\mu}^\nu=c_{\bar\lambda\bar\mu}^{\nu-(\lambda_n+\mu_n)^n}.
\end{equation}

Set 
$$
\Horn_n=\{(\lambda,\mu,\nu)\in(\Lambda_n)^3\,:\,c_{\lambda\mu}^\nu\neq 0\}.
$$

By a Brion-Knop's result (see \cite{Elash}), $\Horn_n$ is a finitely generated semigroup. The
Knutson-Tao saturation theorem \cite{KT:saturation} shows that
$\Horn_n$ is the set of integer points in a convex cone, the Horn
cone. 
The Horn cone is polyhedral and the minimal list of inequalities
defining it is known (see {\it e.g.}
\cite{Fulton:survey,Belk:c1,KTW}). These inequalities contain the Weyl
inequalities
\begin{equation}
  \label{eq:Weyl}
  \nu_{i+j-1}\leq \lambda_i+\mu_j
\quad{\rm\ whenever\ }i+j-1\leq n;
\end{equation}
and are all of the form
\begin{equation}
  \label{eq:Horn}
  \sum_{k\in K}\nu_k\leq \sum_{i\in I}\lambda_i+\sum_{j\in J}\mu_j,
\end{equation}
for some triples $(I,J,K)$ of three subsets of $\{1,\dots,n\}$ of the same
cardinality.

\begin{prop}
\label{prop:nbisquasipol}
Fix $n\geq 0$. The function
$$
\begin{array}{cccl}
 \Nb_n(c_{\lambda\mu}^\bullet> 0)\,:& \Lambda_n\times \Lambda_n&\longto&\NN\\
&(\lambda,\mu)&\longmapsto&\#\{\nu\in \Lambda_n \,:\,c_{\lambda\mu}^\nu> 0\}.
\end{array}
$$  
is piecewise quasi-polynomial.
\end{prop}

This means that there exist a sub-lattice $\Lambda$ of
$\ZZ^{2n}\supset \Lambda_n\times \Lambda_n$ of finite index and a collection of piecewise polynomial
functions (see Footnote~\ref{foot1}) parametrized by the classes in
$\ZZ^{2n}/\Lambda$. Applying to $(\lambda,\mu)$
the  piecewise polynomial function corresponding to its class, one
then gets $\Nb_n(c_{\lambda\mu}^\bullet> 0)$.

\begin{proof}
We have

\begin{equation}
\Nb_n(c_{\lambda\mu}^\bullet> 0)=\# \bigg (\Horn_n\cap
(\{(\lambda,\mu)\}\times\Lambda_n)\bigg).\label{eq:NBaff}
\end{equation}

Consider the Horn cone $\Horn_n^\QQ$ generated by $\Horn_n$. By the discussion above this proposition, it is defined as a subset of $\QQ^{3n}$ by an explicit list of linear inequalities, namely the Horn inequalities~\eqref{eq:Horn}. Knutson-Tao's saturation Theorem \cite{KT:saturation} asserts that $\Horn_n$ is precisely the set of integer points (that is, belonging to $(\Lambda_n)^3$) in $\Horn_n^\QQ$. Now, equality~\eqref{eq:NBaff} describes $\Nb_n(c_{\lambda\mu}^\bullet> 0)$ as the number of integer points in the affine section of the Horn cone obtained by fixing $(\lambda,\mu)$.  

Moreover each inequality~\ref{eq:Horn} of the Horn cone depends linearly on
$(\lambda,\,\mu)$. Now the conclusion is a consequence of the general
theory of multivariate Ehrhart polynomials (see {\it e.g.} \cite[Theorem~1.1]{BBKLV} or \cite{Sturm:vectpart}).
\end{proof}

  

\section{The hive model}

For later use, we shortly review the hive model that expresses the
Littlewood-Richardson coefficients as the number of integer points in
polyhedra.
Fix an integer $n\geq 2$. 

Let $\lambda$, $\mu$ and $\nu$ in $\Lambda_n$ such that
$|\nu|=|\lambda|+|\mu|$.
Otherwise $c_{\lambda\mu}^\nu=0$.
Label the   $\frac{(n+1)(n+2)}2$ vertices of the triangles
on Figure~\ref{fig:hive} with integers, with the labels on the
boundaries determined by $\lambda,\mu$
and $\nu$ as drawn on the left.

\def\nhive{4}
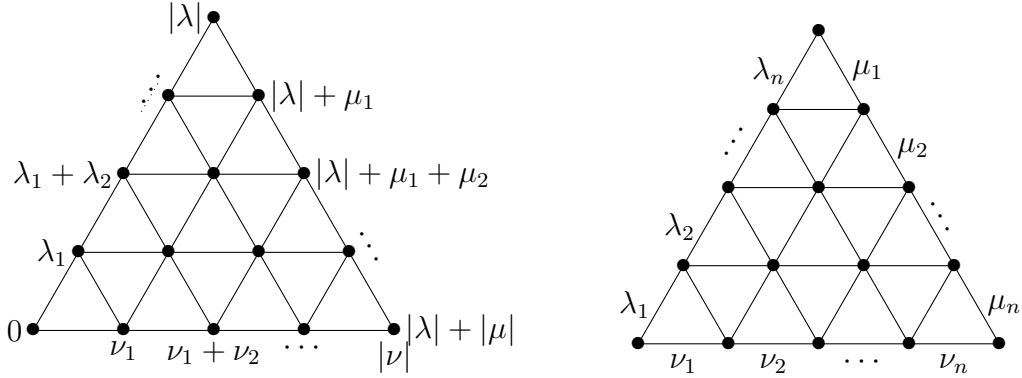
\begin{figure}
\centering
\begin{tikzpicture}[scale=1.2]
\foreach \x in {0,1,...,\nhive} {
\draw (\x,0) -- ++(60:\nhive-\x);
\draw (\x,0) -- ++(120:\x);
\draw (60:\x) -- ++(\nhive-\x,0);
      \foreach \y in {0,...,\x}{
	    \draw (60:\nhive-\x)++(\y,0) node{$\bullet$};
       };
};
\draw (0,0) node[left]{$0$};
\draw (60:1) node[left]{$\lambda_1$};
\draw (60:2) node[left]{$\lambda_1+\lambda_2$};
\draw (60:3) node[left,above,rotate=60]{$\cdots$};
\draw [dotted] (60:2.8)++(-0.2,0) --++(60:0.4);
\draw (60:4) node[left]{$|\lambda|$};
\draw (60:\nhive-1)++(1,0) node[right]{$|\lambda|+\mu_1$};
\draw (60:\nhive-2)++(2,0) node[right]{$|\lambda|+\mu_1+\mu_2$};
\draw (60:1)++(3.05,0.35) node[right,rotate=-60]{$\cdots$};
\draw (0,0)++(\nhive,0) node[right]{$|\lambda|+|\mu|$};
\draw (\nhive,0) node[below]{$|\nu|$};
\draw (1,0) node[below]{$\nu_1$};
\draw (2,0) node[below]{$\nu_1+\nu_2$};
\draw (3,0) node[below]{$\cdots$};
\end{tikzpicture} 
\hspace{2em}
\begin{tikzpicture}[scale=1.2]
\foreach \x in {0,1,...,\nhive} {
\draw (\x,0) -- ++(60:\nhive-\x);
\draw (\x,0) -- ++(120:\x);
\draw (60:\x) -- ++(\nhive-\x,0);
      \foreach \y in {0,...,\x}{
	    \draw (60:\nhive-\x)++(\y,0) node{$\bullet$};
       };
};
\draw (60:0.5) node[left]{$\lambda_1$};
\draw (60:1.5) node[left]{$\lambda_2$};
\draw (60:2.5) node[left,above,rotate=60]{$\cdots$};
\draw (60:\nhive-0.5) node[left]{$\lambda_n$};
\draw (60:\nhive-0.5)++(0.5,0) node[right]{$\mu_1$};
\draw (60:\nhive-1.5)++(1.5,0) node[right]{$\mu_2$};
\draw (60:1.6)++(2.55,0) node[rotate=300]{$\cdots$};
\draw (60:0.5)++(\nhive-0.5,0) node[right]{$\mu_n$};
\draw (\nhive-0.5,0) node[below]{$\nu_n$};
\draw (\nhive-1.5,0) node[below]{$\cdots$};
\draw (0.5,0) node[below]{$\nu_1$};
\draw (1.5,0) node[below]{$\nu_2$};
\end{tikzpicture}
\caption{Hives with boundary conditions}
\label{fig:hive} 
\end{figure}

Then, neighbouring vertices define three distinct types of rhombus (see Figure~\ref{fig:rombus}), each coming with
its own constraint condition on the integers labelling the vertices. For each extracted rhombus we impose the
following constraint (using the notation of Figure~\ref{fig:rombus}):
\begin{equation}
b+c \geq a+d 
\label{HCvertex}
\end{equation} 

Alternatively, one can label the edges of the drawing rather than the vertices: each edge is labelled by the difference between the
values on its vertices, with an orientation as shown on the right of Figure~\ref{fig:hive}. Then each rhombus gives also a constraint on the labels of the edges, equivalent to the previous one on the labels of the vertices (still using the notation of Figure~\ref{fig:rombus}):
\begin{equation}
\beta\geq \delta
\qquad{\rm or\ equivalently}\qquad
\alpha\geq\gamma.
\label{HCvertexgreek}
\end{equation}
The equivalence between these two inequalities comes from the fact that $\alpha+\delta=\beta+\gamma$.

\begin{figure}
\centering
  \begin{tabular}{c@{\hspace{1cm}}c@{\hspace{1cm}}c}
    $R_1$&$R_2$&$R_3$\\
    \begin{tikzpicture}[scale=1]
      \draw (0,0) node{$\bullet$} node[left]{$c$}-- (60:1)
      node{$\bullet$} --++(-1,0) node{$\bullet$} node[left]{$a$}--
      (0,0); \draw (0,0)--(1,0) node{$\bullet$} node[right]{$d$}--
      (60:1) node[right]{$b$}; \draw (0.5,0) node[below]{$\gamma$}
      ++(120:1) node[above]{$\alpha$}; \draw (120:0.5)
      node[left]{$\beta$} ++(1,0) node[right]{$\delta$};
    \end{tikzpicture}
         &
\begin{tikzpicture}[scale=1]
  \draw (1,0) node{$\bullet$} node[right]{$c$}-- (0,0) node{$\bullet$}
  node[left]{$b$}--(60:1) node{$\bullet$} node[above]{$a$} -- cycle;
  \draw (0,0)--(-60:1) node{$\bullet$} node[below]{$d$} -- (1,0);
  \draw (-60:0.5) node[left]{$\gamma$} ++(0.5,0) node[right]{$\beta$};
  \draw (60:0.5) node[left]{$\delta$} ++(0.5,0) node[right]{$\alpha$};
\end{tikzpicture}
               &
                 \begin{tikzpicture}[scale=1]
                   \draw (0,0) node{$\bullet$} node[left]{$a$}--
                   (60:1) node{$\bullet$} node[left]{$b$} --++(1,0)
                   node{$\bullet$} node[right]{$d$}-- (1,0)
                   node{$\bullet$} node[right]{$c$} -- cycle; \draw
                   (1,0) -- (60:1); \draw (0.5,0)
                   node[below]{$\beta$}; \draw (60:0.5)
                   node[left]{$\alpha$} ++(1,0) node[right]{$\gamma$};
                   \draw (60:1)++(0.5,0) node[above]{$\delta$};
                 \end{tikzpicture}
  \end{tabular}
\caption{Rhombi}
\label{fig:rombus}
\end{figure}

By definition a hive is a labelling in $\ZZ^{\frac{(n+1)(n+2)}2}$
satisfying  Inequalities~\eqref{HCvertex} for each one of the $3\frac{n(n-1)}2$
rhombi. The  Knutson-Tao's result is

\begin{theo} (see \cite[Appendix]{KT:saturation})
  Let $\lambda,\mu$ and $\nu$ in $\Lambda_n$. 
Then $c_{\lambda\mu}^\nu$ is the number of hives with boundary conditions
determined by $\lambda$, $\mu$ and $\nu$ as on the left of  Figure~\ref{fig:hive}.  
\end{theo}


\section{The case of $\GL_3$}
\label{sec:SL3}

It is known that the function $(\Lambda_n)^3\longto
\NN,\,(\lambda,\mu,\nu)\longmapsto c_{\lambda\mu}^\nu$ is piecewise
polynomial (see~\cite{Rassart:polLR}) of degree $\frac {n^2-3n+2}2$.
For $n=3$, we get precisely the following. 

\begin{prop}
\label{prop:LR3}
  Let $\lambda=(\lambda_1,\lambda_2,0)$, $\mu=(\mu_1,\mu_2,0)$, and
  $\nu=(\nu_1,\nu_2,\nu_3)$ in $\Lambda_3$ be such that
  $|\nu|=|\lambda|+|\mu|$.  Then $c_{\lambda\mu}^\nu$ is the number of integer points in
the interval
\[
\llbracket\max(\mu_1-\lambda_2,\mu_2,\nu_1-\lambda_1,\mu_1-\nu_3,\nu_2-\lambda_{2},\mu_1+\mu_2-\nu_2),\min(\mu_1,\nu_1-\lambda_2,\mu_1+\mu_2-\nu_3)\rrbracket
. \]
\end{prop}
 
This statement is well known and can easily be checked using
the hive model. Indeed, once $\lambda$, $\mu$ and $\nu$ are fixed, there is only one
interior entry $x$ to choose in order to determine the  hive. 
The solution set to the system of the 9 rhombus inequalities is an
interval from which the interval of the statement is obtained by
translation by $\lambda_1+\lambda_2$.\\

Proposition~\ref{prop:LR3}
implies that, for any nonnegative integer $c$, $c_{\lambda\mu}^\nu>c$
if and only if, for any linear form $\varphi$ and $\psi$ appearing
in the min and max respectively, we have $\varphi-\psi\geq c$.
Namely $c_{\lambda\mu}^\nu>c$ if and only if 

\begin{equation}
\begin{array}{l@{\qquad}l}
  \lambda_1-\lambda_2-c\geq 0&\lambda_2-c\geq 0\\
\mu_1-\mu_2-c\geq 0&\mu_2-c\geq 0\\
\nu_1-\nu_2-c\geq 0&\nu_2-\nu_3-c\geq 0\\
\lambda_1+\mu_1-\nu_1-c\geq 0&\lambda_1+\mu_1-\nu_2-\nu_3-c\geq 0\\
\lambda_1+\mu_2-\nu_2-c\geq 0&\lambda_1+\lambda_2+\mu_1-\nu_1-\nu_3-c\geq 0\\
\lambda_1-\nu_3-c\geq 0&\lambda_1+\lambda_2+\mu_2-\nu_2-\nu_3-c\geq 0\\
\lambda_2+\mu_1-\nu_2-c\geq 0&
	\lambda_1+\mu_1+\mu_2-\nu_1-\nu_3-c\geq 0\\
\mu_1-\nu_3-c\geq 0&
	\lambda_2+\mu_1+\mu_2-\nu_2-\nu_3-c\geq 0\\
\lambda_2+\mu_2-\nu_3-c\geq 0&\lambda_1+\lambda_2+\mu_1+\mu_2-\nu_1-\nu_2-c\geq 0
\end{array}\label{eq:18ineq}
\end{equation}


and
\begin{equation}
  \label{eq:1}
  |\nu|=|\lambda|+|\mu|.
\end{equation}

Note that, for $c=0$, we recover the 6 inequalities saying that
$\lambda$, $\mu$ and $\nu$ are dominant, the 6 Weyl inequalities and
the 6 others inequalities of the Horn cone (see {\it e.g.}
\cite{Fulton:survey}). 

We now want to compute the function mapping $(\lambda,\mu,c)$ to the
number of solutions of this system of inequalities in $\nu$. 
The method consists in restating this problem in the langage of vector partition
functions as in \cite{Sturm:vectpart}.

Start with the $18\times 8$ matrix $H$  whose rows give the coefficients of the 18 inequalities~\eqref{eq:18ineq}.
 Set 
$$
\Lambda=\{(\lambda_1,\lambda_2,\mu_1,\mu_2,\nu_1,\nu_2,\nu_3,c)\in \ZZ^8\,:\, |\nu|=|\lambda|+|\mu|\}
$$
and
$$
\Lambda^+=\{(\lambda_1,\lambda_2,\mu_1,\mu_2,\nu_1,\nu_2,\nu_3,c)\in
\Lambda\,:\, \lambda,\mu,\nu\mathrm{\ dominant\ and\ } c\geq 0\}.
$$
Let $\widetilde\Horn_3$ denote the set of points in $\Lambda^+$ that satisfy
the inequalities~\eqref{eq:18ineq}.  

To get nonnegative variables, we make the following change of
coordinates 
$$
\begin{array}{lll}
  a_1=\lambda_1-\lambda_2-c&b_1=\mu_1-\mu_2-c&c_1=\nu_1-\nu_2-c\\
a_2=\lambda_2-c&b_2=\mu_2-c&c_2=\nu_2-c
\end{array}
$$ 
Then $\widetilde\Horn_3$ identifies with $\widetilde \Horn_3'=\{X\in\NN^7\,|\,NX\geq 0\}$, where 

$$
N=
\left(\begin{array}{rrrrrrr}
-1 & -2 & -1 & -2 & 1 & 3 & -3 \\
1 & 1 & 1 & 1 & -1 & -1 & 1 \\
1 & 1 & 0 & 1 & 0 & -1 & 1 \\
0 & -1 & -1 & -2 & 1 & 2 & -2 \\
0 & 1 & 1 & 1 & 0 & -1 & 1 \\
-1 & -2 & 0 & -1 & 1 & 2 & -2 \\
-1 & -1 & -1 & -1 & 1 & 2 & -2 \\
0 & -1 & 0 & -1 & 1 & 1 & -1 \\
0 & 0 & 0 & -1 & 0 & 1 & -1 \\
0 & 0 & -1 & -1 & 1 & 1 & -1 \\
0 & -1 & 0 & 0 & 0 & 1 & -1 \\
-1 & -1 & 0 & 0 & 1 & 1 & -1 \\
1 & 2 & 1 & 2 & -1 & -2 & 2
\end{array}\right).
$$

Set $
\tilde N=(N | -I_{13})
$ in such a way that 
$$
\begin{array}{lll}
\widetilde \Horn_3'
&\simeq&\{(X,Y)\in\NN^7\times \NN^{13}\,|\,NX=Y\}\\
&\simeq&\{X\in \NN^{20}\,|\,\tilde NX=0\}.
\end{array}
$$
We now proceed to the affine section mentioned in the proof of Proposition~\ref{prop:nbisquasipol}.
Thus, up to our changes of variables, the function
$(\lambda,\mu,c)\mapsto\Nb_3(c^\bullet_{\lambda\mu}>c)$  is the   map
$$
\begin{array}{ccl}
  \NN^5 &\longto&\NN\\
Y&\longmapsto&\#\{X\in\NN^{15}\,:\,\tilde MX=-BY\},
\end{array}
$$
where $\tilde M=(M\,|\,-I_{13})$,
$M$ is the matrix formed by columns 5 and 6 of the
matrix $N$ and 
$B$  is the matrix formed by  the other columns of $N$.

Note that $\tilde M$ is not unimodular: the lowest common multiple of the maximal
minors is not $1$, but $6$. There are 83 such nonzero minors. 
Then \cite{Sturm:vectpart} implies that $(\lambda,\mu,c)\mapsto
\Nb(c_{\lambda\mu}^\bullet>c)$ 
is piecewise quasi-polynomial with chambers obtained by intersecting some 
83 explicit simplicial cones. 
We used \cite{iscc}, an implementation of Barvinok algorithm \cite{Barv}
to compute this function. The surprise was that we got only polynomial
functions and only 7 cones. 
Actually, the software gave 36 cones that can be glued to give those 7
described in Proposition~\ref{prop:calculSL3}.

\begin{prop}
\label{prop:calculSL3}
  We use the basis of fundamental weights to set
  $\lambda=k_1\varpi_1+k_2\varpi_2=(k_1+k_2,k_2)$ and $\mu=l_1\varpi_1+l_2\varpi_2=(l_1+l_2,l_2)$.
 Then $\Nb(c_{\lambda\mu}^\bullet>c)=0$ unless
$$
c\leq\min(k_1,k_2,l_1,l_2).
$$
Moreover, this cone decomposes into 7 cones $C_1,\dots,C_7$ on which $\Nb(c_{\lambda\mu}^\bullet>c)$ is given by
polynomial functions $P_1,\dots,P_7$. Five of these seven pairs
$(C_i,P_i)$ are kept unchanged by switching $k_1$ and $k_2$. The two others are swapped by this operation.

In particular, Conjecture~\ref{conj:setLR} holds for $\GL_3$. 
\end{prop}

In the basis of fundamental weights, we are interested in the function
\[ \begin{array}{rccl}
\psi: & \NN^5 & \longrightarrow & \NN\\
 & (k_1,k_2,l_1,l_2,c) & \longmapsto & \#\{\nu\in\Lambda_3\,|\,c_{k_1\varpi_1+k_2\varpi_2,l_1\varpi_1+l_2\varpi_2}^\nu >c\}
\end{array}. \]
Notice moreover that switching $k_1$ and $k_2$ corresponds then to
taking $\lambda^*$. 
Define now the following seven polynomials in $k_1,k_2,l_1,l_2,c$:
\[ P_1=2 \, c^{2} - c {\left(k_{1} + k_{2} + l_{1} + l_{2} + 2\right)}
  - \frac{1}{2} (k_1+k_2-l_1-l_2)^2+k_1k_2+l_1l_2+\frac 1
  2(k_1+k_2+l_1+l_2)
+1\]
\[ P_2=3c^{2} - 3c {\left(k_{1} + k_{2} + 1\right)} + 
\frac{1}{2}(k_{1}+k_2)^{2} +  k_{1}k_{2}+\frac 3 2(k_1+k_2)
+1 \]
\[ P_3=3c^{2} - 3c {\left(l_{1} + l_{2} + 1\right)} + 
\frac{1}{2}(l_{1}+l_2)^{2} +  l_{1}l_{2}+\frac 3 2(l_1+l_2)
+1 \]
\[ P_4=\frac{5}{2} \, c^{2} - c {\left(2 \, k_{1} + 2
      \, k_{2} +  l_{1} + \frac 5 2\right)} + k_{1} k_{2} +(k_{1} + k_{2})(l_{1}+1) - \frac{l_1}{2}( l_{1}-1) + 1 \]
\[ P_5=\frac{5}{2} \, c^{2} - c {\left(2 \, k_{1} + 2
      \, k_{2} +  l_2+ \frac 5 2\right)} + k_{1} k_{2} +(k_{1} +
  k_{2})(l_2+1) - \frac{l_2}{2}( l_2-1) + 1 \]

\[ P_6=\frac{5}{2} \, c^{2} - c {\left(k_{1} + 2 l_{1} + 2 l_{2}
      +\frac 5 2\right)} +l_1l_2+(l_1+l_2)(k_1+1)-\frac{k_1}{2}( k_{1}-1) + 1\]
\[ P_7=\frac{5}{2} \, c^{2} - c {\left(k_2 + 2 l_{1} + 2 l_{2}
      +\frac 5 2\right)} +l_1l_2+(l_1+l_2)(k_2+1)-\frac{k_2}{2}(
  k_2-1) + 1\]

Notice  that $P_1,\dots,P_5$ are symmetric in $k_1,k_2$,
whereas $P_6$ and $P_7$ are swapped when one swaps $k_1$ and
$k_2$. One might also add that $P_3$ is the image of $P_2$ under the
involution corresponding to swapping  $\lambda$ and
$\mu$ -- i.e. swapping $(k_1,k_2)$ and $(l_1,l_2)$ --, as $P_6$ is
the image of $P_4$ and $P_7$ the one of $P_5$ under this same
involution. 

Then, for $c\geq 0$, $k_1\geq c$, $k_2\geq c$, $l_1\geq c$, and $l_2\geq c$, the function $\psi$ is given by the following piecewise polynomial:
\[ \begin{array}{l|l}
\text{Cones of polynomiality} & \text{Polynomial giving }\psi\\
\hline
C_1:k_1+k_2\geq \max(l_1,l_2)+c,\quad l_1+l_2\geq \max(k_1,k_2)+c & P_1\\
C_2:k_1+k_2\leq \min(l_1,l_2)+c& P_2\\
C_3:l_1+l_2\leq \min(k_1,k_2)+c& P_3\\
C_4:l_1+c\leq k_1+k_2\leq l_2+c& P_4\\
C_5: l_2+c\leq k_1+k_2\leq l_1+c& P_5\\
C_6:k_1+c\leq l_1+l_2\leq k_2+c & P_6\\
C_7: k_2+c\leq l_1+l_2\leq k_1+c  & P_7
\end{array} \]
One can then see that the cones $C_1$ to $C_5$ are stable under permutation of $k_1$ and $k_2$ whereas the cones $C_6$ and $C_7$ are swapped when $k_1$ and $k_2$ are. Thus, for all $k_1,k_2,l_1,l_2,c\geq 0$,
\[ \psi(k_1,k_2,l_1,l_2,c)=\psi(k_2,k_1,l_1,l_2,c), \]
which proves Proposition~\ref{prop:calculSL3}.\\

\begin{rem}
The last part of  Proposition~\ref{prop:calculSL3} asserts that
there exists a bijection
$(\Lambda_3^0)^2\times\Lambda_3\longto (\Lambda_3^0)^2\times\Lambda_3,\,(\lambda,\mu,\nu)\longmapsto (\lambda^*,\mu,\tilde\nu)$ such that
$$
c_{\lambda\mu}^\nu=c_{\lambda^*\mu}^{\tilde \nu}.
$$
One could hope for such a bijection to be linear. Unfortunately, it CANNOT.\\

One can check this claim as follows. The matrix of such a linear bijection $\varphi:(\Lambda_3^0)^2\times\Lambda_3\longto (\Lambda_3^0)^2\times\Lambda_3,\,(\lambda,\mu,\nu)\longmapsto (\lambda^*,\mu,\tilde\nu)$ would only depend on 7 vectors $\nu^1,\dots,\nu^7$ in $\ZZ^3$. Assume that this matrix does exist.

For example, the image of $(\lambda_1,\lambda_2,\mu_1,\mu_2,\nu_1,\nu_2,\nu_3)=(1,0,0,0,1,0,0)$ is then $(1,1,0,0,\nu^1+\nu^5)$. Moreover this image has to correspond to some nonzero Littlewood-Richardson coefficient. Hence it is necessary that $\nu_1+\nu_5=(1,1,0)$.

Similarly, considering the image of $(0,0,1,1,1,1,0)$, one gets that it must be $(0,0,1,1,\nu^3+\nu^4+\nu^5+\nu^6)=(0,0,1,1,1,1,0)$, and then $\nu^3+\nu^4+\nu^5+\nu^6=(1,1,0)$.

Now the image of $(1,0,1,1,1,1,1)$ has to be a ray of the Horn cone. We deduce that this image is $(1,1,1,1,2,2,0)$, and thus $\nu^1+\nu^3+\nu^4+\nu^5+\nu^6+\nu^7=(2,2,0)$.

Combining these three constraints one gets $\nu^5=\nu^7$, which contradicts the invertibility of $\varphi$.\\

Note also that the linear automorphisms of $(\Lambda_3)^3$ preserving
the Littlewood-Richardson coefficients are proved to form a group of cardinality
288 (so big !) in \cite{bri144}.
\end{rem}

\section{A stability result}

In this section, we are interested in the Littlewood-Richardson
coefficients $c_{\lambda\mu}^\nu$ with $\lambda$ and $\mu$
near-rectangular. Using the hive model, we give a proof of
Proposition~\ref{prop:stab} stated in the introduction.

\begin{proof}[Proof of Proposition~\ref{prop:stab}]
We know that the Littlewood-Richardson coefficient $c_{\lambda\mu}^\nu$ is equal to the number of hives with exterior edges labelled by $\lambda$, $\mu$, and $\nu$. Let us assume that $c_{\lambda\mu}^\nu>0$, meaning that such a hive exists, and consider any hive like this. In this hive, a certain number of edges have a label that is already fixed by these ``boundary conditions''. These edges and values are shown in the following picture (made for the sake of example with a hive of size 6, but the picture is strictly the same for all sizes at least 4):
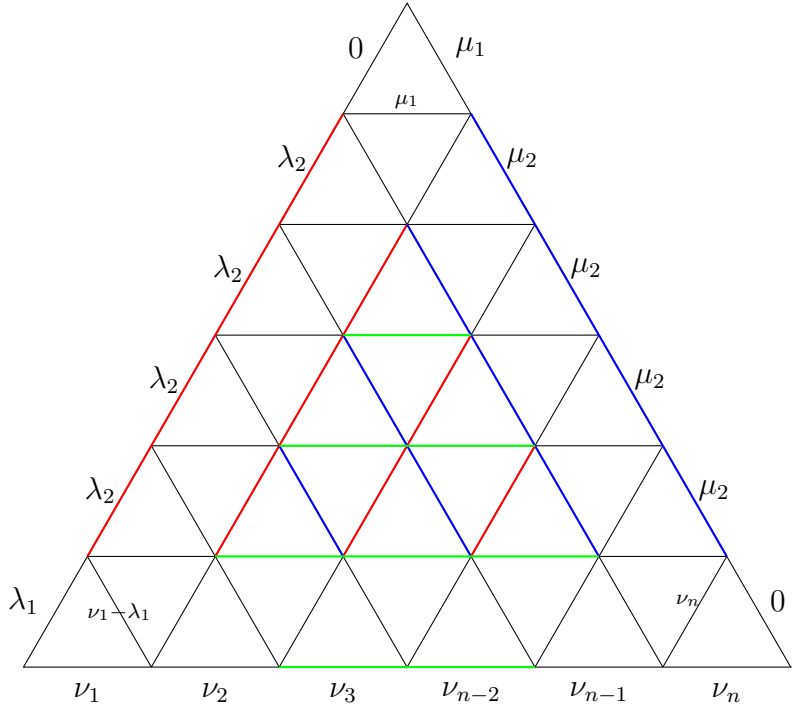
\def\nhive{6}
\begin{center}
\begin{tikzpicture}[scale=1.7]
\foreach \x in {0,1,...,\nhive} {
\draw (\x,0) -- ++(60:\nhive-\x);
\draw (\x,0) -- ++(120:\x);
\draw (60:\x) -- ++(\nhive-\x,0);
}
\foreach \x in {3,4,...,\nhive} {
\draw[thick,red] (60:1) ++(\x-3,0) -- ++(60:\nhive+1-\x);
\draw[thick,blue] (60:1) ++(\nhive+2-\x,0) -- ++(120:\nhive+1-\x);
}
\foreach \x in {4,5,...,\nhive} {
\draw[thick,green]  (60:1+\x-4) ++(1,0) -- ++(\nhive+1-\x,0);
}
\draw[thick,green]  (2,0) -- ++(\nhive-4,0);
\node at (0.5,-0.2) {$\nu_1$};
\node at (1.5,-0.2) {$\nu_2$};
\node at (2.5,-0.2) {$\nu_3$};
\node at (3.5,-0.2) {$\nu_{n-2}$};
\node at (\nhive-1.5,-0.2) {$\nu_{n-1}$};
\node at (\nhive-0.5,-0.2) {$\nu_n$};
\path (60:0.6) ++(-0.3,0)
  node {$\lambda_1$};
\path (60:\nhive-0.4) ++(-0.2,0)
  node {$0$};
\foreach \x in {3,4,...,\nhive} {
\path (60:\x-1.4) ++(-0.2,0)
  node {$\lambda_2$};
}
\path (\nhive,0) ++(120:0.6) ++(0.2,0)
  node {$0$};
\path (\nhive,0) ++(120:\nhive-0.4) ++(0.3,0)
  node {$\mu_1$};
\foreach \x in {3,4,...,\nhive} {
\path (\nhive,0) ++(120:\x-1.4) ++(0.2,0)
  node {$\mu_2$};
}
\path (1,0) ++(120:0.5)
  node {$\scriptstyle \nu_1-\lambda_1$};
\path (\nhive-1,0) ++(60:0.6) ++(-0.1,0)
  node {$\scriptstyle\nu_n$};
\path (60:\nhive-1) ++(0.5,0.1)
  node {$\scriptstyle \mu_1$};
\end{tikzpicture} 
\end{center}

Then, using a number of hive conditions (including the fact that, in any triangle, the values on two of the edges determine the value on the third one), we obtain the following values for some interior edges: all the edges parallel to the north-west side of the triangle, with no endpoint on the other two sides, correspond to the value $\lambda_2$; the ones parallel to the north-east side, with no endpoint on the other two sides, have the value $\mu_2$; finally the edges parallel to the south side, strictly inside the triangle and with no endpoint on the other two sides, must have the value $\lambda_2+\mu_2$ \footnote{These three kinds of edges are coloured on both pictures in this proof: the first kind (parallel to the north-west) in red, the second kind (north-east) in blue, and the third one (south) in green.}.

Thus the edges on the south side corresponding to the parts $\nu_3$ to $\nu_{n-2}$ must in fact have the value $\lambda_2+\mu_2$. That is, $\nu$ must be of the aforementioned form: $\nu=\nu_1\nu_2(\lambda_2+\mu_2)^{n-4}\nu_{n-1}\nu_n$. Notice moreover that, even if $n=4$, two rhombi inequalities show immediately that one must still have $\nu_2\geq\lambda_2+\mu_2\geq\nu_{n-1}$.\\

From now on we assume that $\nu$ has this particular form. Consider once again a hive with exterior edges labelled by $\lambda$, $\mu$, and $\nu$. The same conditions as before apply, and the values of all the remaining edges have then to be chosen in order to determine completely the hive. Let us name these values $a_0,a_1,\dots,a_7$ as shown on the following picture:
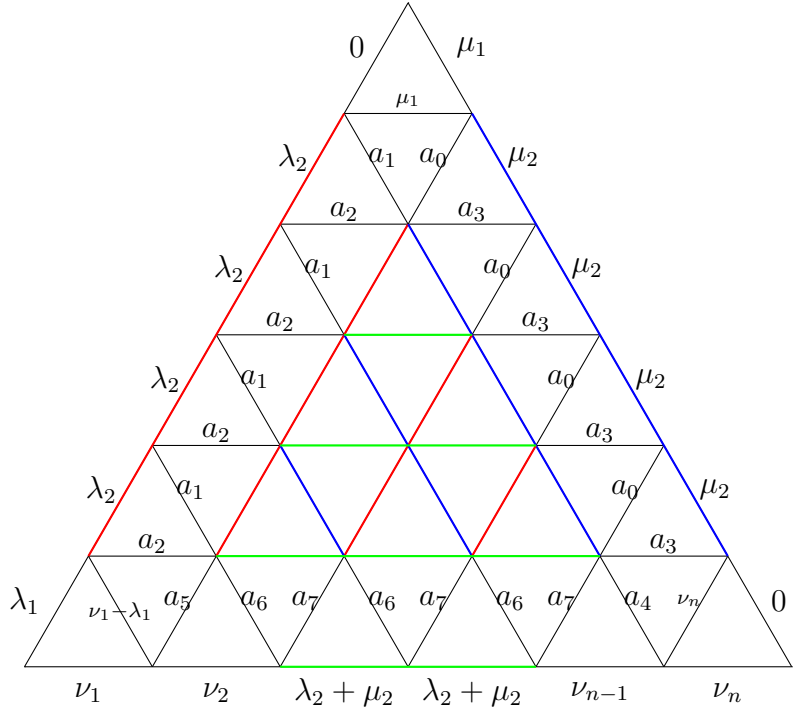
\begin{center}
\begin{tikzpicture}[scale=1.7]
\foreach \x in {0,1,...,\nhive} {
\draw (\x,0) -- ++(60:\nhive-\x);
\draw (\x,0) -- ++(120:\x);
\draw (60:\x) -- ++(\nhive-\x,0);
}
\foreach \x in {3,4,...,\nhive} {
\draw[thick,red] (60:1) ++(\x-3,0) -- ++(60:\nhive+1-\x);
\draw[thick,blue] (60:1) ++(\nhive+2-\x,0) -- ++(120:\nhive+1-\x);
}
\foreach \x in {4,5,...,\nhive} {
\draw[thick,green]  (60:1+\x-4) ++(1,0) -- ++(\nhive+1-\x,0);
}
\draw[thick,green]  (2,0) -- ++(\nhive-4,0);
\foreach \x in {3,4,...,\nhive} {
\path (60:\x-1.4) ++(\nhive-\x+0.9,0)
  node {$a_0$};
}
\foreach \x in {3,4,...,\nhive} {
\path (60:\x-0.9) ++(-60:0.5)
  node {$a_1$};
}
\foreach \x in {3,4,...,\nhive} {
\path (60:\x-2) ++(0.5,0.1)
  node {$a_2$};
}
\foreach \x in {3,4,...,\nhive} {
\path (60:\x-2) ++(\nhive-\x+1.5,0.1)
  node {$a_3$};
}
\path (\nhive-0.9,0) ++(120:0.6)
  node {$a_4$};
\path (0.9,0) ++(60:0.6)
  node {$a_5$};
\foreach \x in {4,5,...,\nhive} {
\path (60:1.1) ++(\x-3,0) ++(-60:0.5)
  node {$a_6$};
}
\foreach \x in {4,5,...,\nhive} {
\path (60:0.6) ++(\x-2.1,0)
  node {$a_7$};
}
\node at (0.5,-0.2) {$\nu_1$};
\node at (1.5,-0.2) {$\nu_2$};
\node at (\nhive-1.5,-0.2) {$\nu_{n-1}$};
\node at (\nhive-0.5,-0.2) {$\nu_n$};
\foreach \x in {5,6,...,\nhive} {
\node at (\x-2.5,-0.2) {$\lambda_2+\mu_2$};
}
\path (60:0.6) ++(-0.3,0)
  node {$\lambda_1$};
\path (60:\nhive-0.4) ++(-0.2,0)
  node {$0$};
\foreach \x in {3,4,...,\nhive} {
\path (60:\x-1.4) ++(-0.2,0)
  node {$\lambda_2$};
}
\path (\nhive,0) ++(120:0.6) ++(0.2,0)
  node {$0$};
\path (\nhive,0) ++(120:\nhive-0.4) ++(0.3,0)
  node {$\mu_1$};
\foreach \x in {3,4,...,\nhive} {
\path (\nhive,0) ++(120:\x-1.4) ++(0.2,0)
  node {$\mu_2$};
}
\path (1,0) ++(120:0.5)
  node {$\scriptstyle \nu_1-\lambda_1$};
\path (\nhive-1,0) ++(60:0.6) ++(-0.1,0)
  node {$\scriptstyle \nu_n$};
\path (60:\nhive-1) ++(0.5,0.1)
  node {$\scriptstyle \mu_1$};
\end{tikzpicture} 
\end{center}

The fact that many of these edges must be given the same value comes everytime from the fact that, in any rhombus inside a hive, if two opposite edges have the same value, then it must also be the case for the other pair of opposite edges.\\

Using now once again the fact that, in any triangle, the values corresponding to two edges determine the value on the third one, these 8 integers $a_0,\dots,a_7$ are  related by the following equations:
\[ \left\lbrace\begin{array}{l}
a_0+a_1=\mu_1\\
a_0+\mu_2=a_3\\
\lambda_2+a_1=a_2\\
\nu_1-\lambda_1+a_5=a_2\\
a_4+\nu_n=a_3\\
a_4+a_7=\nu_{n-1}\\
a_5+a_6=\nu_2\\
a_6+a_7=\lambda_2+\mu_2
\end{array}\right.\qquad\Longleftrightarrow\qquad\left\lbrace\begin{array}{l}
a_1=\mu_1-a_0\\
a_2=\lambda_2+\mu_1-a_0\\
a_3=\mu_2+a_0\\
a_4=\mu_2-\nu_n+a_0\\
a_5=\lambda_1+\lambda_2+\mu_1-\nu_1-a_0\\
a_6=\lambda_2+2\mu_2-\nu_{n-1}-\nu_n+a_0\\
a_7=-\mu_2+\nu_{n-1}+\nu_n-a_0
\end{array}\right. \]
(let us recall that the equality $|\lambda|+|\mu|=|\nu|$ means that $\lambda_1+2\lambda_2+\mu_1+2\mu_2=\nu_1+\nu_2+\nu_{n-1}+\nu_n$). In particular, this means that the hive is for instance entirely determined by the value of $a_0$. We can now look at all the hive inequalities that must be satisfied by these $a_i$'s:
\[ \begin{array}{rcl}
a_0 & \geq & 0\\
\lambda_2 & \geq & a_0\\
a_1 & \geq & \mu_2
\end{array}\qquad\qquad\begin{array}{rcl}
\nu_1 & \geq & a_2\\
a_1 & \geq & \nu_1-\lambda_1\\
a_2 & \geq & \nu_2\\
\mu_2 & \geq & a_6
\end{array}\qquad\qquad\begin{array}{rcl}
a_3 & \geq & \nu_n\\
\nu_n & \geq & a_0\\
\nu_{n-1} & \geq & a_3
\end{array} \]
The inequalities of the first column can be obtained from the hive
inequalities in the north corner of the hive, those of the second from
the south-west corner, and those of the third from the south-east
corner (keep in mind that some of them can of course be obtained in
several ways). Thanks to the previous relations between the $a_i$'s,
all these inequalities can be expressed in terms of $a_0$ only, giving
in the end exactly the following necessary and sufficient conditions
on $a_0$ to obtain a hive:
\[ \left\lbrace
\begin{array}{l}
a_0\geq \max(0, \lambda_2+\mu_1-\nu_1,\nu_n-\mu_2)\\
a_0\leq \min(\lambda_2, \mu_1-\mu_2, \lambda_1+\mu_1-\nu_1, \lambda_2+\mu_1-\nu_2, -\lambda_2-\mu_2+\nu_{n-1}+\nu_n,\nu_n, \nu_{n-1}-\mu_2)
\end{array}\right. \]
As a consequence, $c_{\lambda\mu}^\nu$ is the number of
integer points in the interval $[M;m]$ where
\[
  \begin{array}{l}
M=\max(0,\lambda_2+\mu_1-\nu_1,\nu_n-\mu_2),\quad{\rm and}\\
m=\min(\lambda_2,\mu_1-\mu_2,\lambda_1+\mu_1-\nu_1,\lambda_2+\mu_1-\nu_2,-\lambda_2-\mu_2+\nu_{n-1}+\nu_n,\nu_n,\nu_{n-1}-\mu_2).
  \end{array}
\]
Observe finally that $\nu_{n-1}\leq \lambda_2+\mu_2$ and $\nu_2\geq \lambda_2+\mu_2$ give
\[ \nu_{n-1}-\mu_2\leq \lambda_2,\qquad \lambda_2+\mu_1-\nu_2\leq \mu_1-\mu_2,\qquad -\lambda_2-\mu_2+\nu_{n-1}+\nu_n\leq \nu_n. \]
Hence $c_{\lambda\mu}^\nu$ is the cardinality of 
\[ \llbracket\max(0,\lambda_2+\mu_1-\nu_1,-\mu_2+\nu_n),\min(\lambda_1+\mu_1-\nu_1,\lambda_2+\mu_1-\nu_2,-\lambda_2-\mu_2+\nu_{n-1}+\nu_n,-\mu_2+\nu_{n-1})\rrbracket. \]
\end{proof}

This stability can be interpreted as a proof of the existence of a bijection between 
sets of hives. Such a bijection can for instance be obtained as follows.

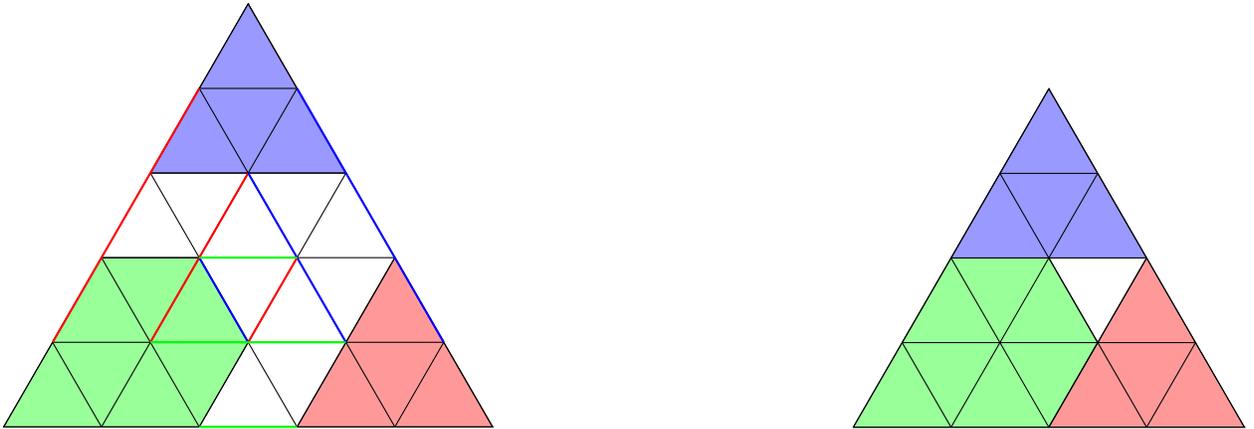
\begin{center}
\def\nhive{5}
\begin{tikzpicture}[scale=1.3]
\draw[fill=green!40] (0,0) -- (2,0) -- ++(60:1) -- ++(120:1) --
++(-1,0) -- cycle;
\draw[fill=red!40] (3,0) -- (5,0) -- ++(120:2) -- cycle;
\draw[fill=blue!40] (60:3) -- ++(2,0) -- (60:5) -- cycle;

\foreach \x in {0,1,...,\nhive} {
\draw (\x,0) -- ++(60:\nhive-\x);
\draw (\x,0) -- ++(120:\x);
\draw (60:\x) -- ++(\nhive-\x,0);
}  
\foreach \x in {0,1,...,2} {
\draw[thick,red]  (60:1) ++(\x,0) -- ++(60:\nhive-2-\x);
\draw[thick,blue]  (60:1) ++(4-\x,0) -- ++(120:\nhive-2-\x);
}  
\foreach \x in {0,1} {
\draw[thick,green]  (2,0) ++(120:\x) -- ++(\x+1,0);
}
\draw[thick,green]  (60:2) ++(1,0) -- ++(1,0);
\end{tikzpicture}
\hfill 
\def\nhive{4}
\begin{tikzpicture}[scale=1.3]
\draw[fill=green!40] (0,0) -- (2,0) -- ++(60:1) -- ++(120:1) --
++(-1,0) -- cycle;
\draw[fill=red!40] (2,0) -- (4,0) -- ++(120:2) -- cycle;
\draw[fill=blue!40] (60:2) -- ++(2,0) -- (60:4) -- cycle;
\foreach \x in {0,1,...,\nhive} {
\draw (\x,0) -- ++(60:\nhive-\x);
\draw (\x,0) -- ++(120:\x);
\draw (60:\x) -- ++(\nhive-\x,0);
}  
\end{tikzpicture} 
\end{center}

Starting from a hive of size $n$ ($n\geq 4$), consider the three areas
coloured in the picture above (on the left): the four triangles in the
north corner, the four in the south-east one, and the seven in the
south-west one. Then send this hive to the one of size $4$ obtained
by keeping these three coloured-areas (picture on the right). The
rhombus inequalities in the hive show that the values on the edges in
these three particular areas determine indeed completely the
hive. This means that this map is well defined and that it is truly a bijection.

\bigskip
\begin{rem}
Let $\alpha,\beta,\gamma$ be three partitions such that
$c_{\alpha,\beta}^\gamma=1$. By the Fulton's conjecture (see
\cite{KT:saturation} or \cite{belkale:geomHorn,BKR,fulton}),
we have $c_{k\alpha,k\beta}^{k\gamma}=1$, for any $k\geq 0$.
Let $(\tilde\alpha,\tilde\beta,\tilde\gamma)$ be a second triple of
partitions.
The stability result of \cite{SamSnowden} (see also \cite{Par:stab,Pel:stab}) asserts that
$c_{\tilde\alpha +k\alpha,
  \tilde\beta+k\beta}^{\tilde\gamma+k\gamma}$ does not depend on the
integer $k$
big enough.

Returning to the setting of Proposition~\ref{prop:stab}, consider
$\alpha=1^{\lambda_2}$, $\beta=1^{\mu_2}$ and
$\gamma=1^{\lambda_2+\mu_2}$. Set also
$\tilde\alpha=(\lambda_1\lambda_2\lambda_2)'$,
$\tilde\beta=(\mu_1\mu_2)'$ and $\tilde
\gamma=(\nu_1\nu_2\nu_3\nu_4)'$, where $\square'$ denote the conjugated
partition.  Using the invariance of the Littlewood-Richardson
coefficient by simultaneous conjugation we get
$c_{\lambda\mu}^\nu=c_{\tilde\alpha +k\alpha,
  \tilde\beta+k\beta}^{\tilde\gamma+k\gamma}$ with $k=n-2$.
Thus, the mentioned stability
result asserts that
$c_{\lambda\mu}^\nu$ does not depend on $n$ big anough.
 Proposition~\ref{prop:stab} asserts that this sequence is in fact constant for $n\geq 2$.
\end{rem}

\section{The case of $\GL_4(\CC)$}

This section is about  $\GL_4(\CC)$. But
Proposition~\ref{prop:stab} allows to extend several results to any
$\GL_n(\CC)$ for $n\geq 4$.

\subsection{The Horn cone}

The set of points in $\Horn_n$ with $\lambda$ and/or $\mu$
near-rectangular is the set of integer points on a face of this cone.
Proposition~\ref{prop:stab} implies that the geometry of this
face and the Littlewood-Richardson coefficients on it do not depend on $n\geq 4$.    
We denote by ${\overline\Horn}_n$ the set of points in $\Horn_n$ with
the first two partitions  $\lambda$ and $\mu$ in $\Lambda_n^0$. Then 
$\Horn_n\simeq \ZZ^2\times {\overline\Horn}_n$.
Set
$$
\Horn_4^{\nr^2}=\{(\lambda,\mu,\nu)\in\Horn_4\,:\,\lambda\mbox{ and
  $\mu$ are near-rectangular}\}
$$ 
and
$$
\Horn_4^{\nr}=\{(\lambda,\mu,\nu)\in\Horn_4\,:\,\lambda
 \mbox{  is near-rectangular}\}.
$$ 
The inequalities defining the Horn cone  $\Horn_n^\QQ$ are
well known (see Section~\ref{sec:Horncone}). 
By convex geometry and explicit computation, one can
deduce the minimal lists of inequalities for $\Horn_4^{\nr^2}$ and
$\Horn_4^{\nr}$. Softwares like Normaliz \cite{normaliz} allow to make the
computation. 
 
\begin{prop}
\label{prop:Horn4nr2}
Let $\lambda, \mu$ in $\Lambda_4^0$ and $\nu$ in $\Lambda_4$ such that $\lambda$ and
$\mu$ are near-rectangular. Then $c_{\lambda\mu}^\nu\neq 0$ if and only if
  $$
  \begin{array}{c@{\ \qquad\ }c}
   \multicolumn{2}{c}{|\lambda|+|\mu|=|\nu|}\\
 \nu_1\geq \nu_2&\nu_4\geq 0\\
\multicolumn{2}{c}{\nu_3+\nu_4\geq\lambda_2+\mu_2}\\
\nu_1+\nu_3\geq \lambda_1+\lambda_2+\mu_2&\nu_1+\nu_3\geq \lambda_2+\mu_1+\mu_2\\
\multicolumn{2}{c}{\nu_2\geq \lambda_2+\mu_2\geq \nu_3}\\
\nu_3\geq \lambda_2&\nu_3\geq \mu_2\\
\lambda_1+\mu_2\geq \nu_2&\lambda_2+\mu_1\geq \nu_2  
  \end{array}
$$
\end{prop}

\begin{rem}
  Proposition~\ref{prop:stab} also implies that $\nu_1+\nu_4\geq
  \lambda_2+\mu_1$, which is a consequence of these 11 inequalities.  
\end{rem}

\begin{prop}
  The cone generated by $\Horn_4^{\nr^2} \cap {\overline\Horn}_4$ has
  8 extremal rays generated by the triples $(\lambda,\mu,\nu)$
  associated to the following inclusions
  \begin{enumerate}
  \item $V(1)\subset V(1)\otimes V(0)$ (twice by permuting the 
    factors);
\item $V(1^3)\subset V(1^3)\otimes V(0)$ (twice by permuting the
    factors);
\item $V(1^2)\subset V(1)\otimes V(1)$;
\item $V(1^4)\subset V(1)\otimes V(1^3)$ (twice by permuting the 
    factors);
\item $V(2^21^2)\subset V(1^3)\otimes V(1^3)$.
  \end{enumerate}
Each triple $(\lambda,\mu,\nu)$ on one of these  extremal rays indexes a Littlewood-Richardson coefficient with value 1. 
The Hilbert basis of $\Horn_4^{\nr^2}\cap {\overline\Horn}_4$ consists in these 8 generators.
\end{prop}

We get similar descriptions for $\Horn_4^\nr$.

\begin{prop}
\label{prop:horn4nr}
Let $\lambda, \mu$ in $\Lambda_4^0$ and $\nu$ in
  $\Lambda_4$ such that $\lambda$ is near-rectangular. 
Then $c_{\lambda\mu}^\nu\neq 0$ if and only if all of the following inequalities hold:
  $$
  \begin{array}{ll}
   \multicolumn{2}{c}{|\lambda|+|\mu|=|\nu|}\\
 
    \multicolumn{2}{c}{\nu_1\geq \nu_2\geq\nu_3\geq\nu_4 \geq 0}\\
\lambda_1\geq\lambda_2\geq 0
&
\mu_1\geq\mu_2\geq\mu_3\geq 0\\

\lambda_1+\mu_1\geq\nu_1
&
\min(\lambda_2+\mu_1, \lambda_1+\mu_2)\geq\nu_2\\

\min(\lambda_2+\mu_2, \lambda_1+\mu_3)\geq\nu_3
&
\min(\lambda_1,\mu_1, \lambda_2+\mu_3)\geq \nu_4\\
  \multicolumn{2}{c}{\nu_1\geq  \max(\lambda_1,\mu_1\lambda_2+\mu_2)\qquad
\nu_2\geq \max(\mu_2,\lambda_2+\mu_3)\qquad
\nu_3\geq\max(\lambda_2,\mu_3)}\\

\nu_1+\nu_2\geq\max(\lambda_1+\lambda_2+\mu_2,
\lambda_2+\mu_1+\mu_2)
 &
\nu_1+\nu_3\geq\max(\lambda_1+\lambda_2+\mu_3,
 \lambda_2+\mu_1+\mu_3)\\
  
\nu_2+\nu_3\geq \lambda_2+\mu_2+\mu_3
 &
\nu_1+\nu_4\geq\lambda_2+\mu_1\\
\nu_2+\nu_4\geq\lambda_2+\mu_2
 &
\nu_3+\nu_4\geq\lambda_2+\mu_3
  \end{array}
$$
We have 32 facets.
\end{prop}

\begin{prop}
  The cone generated by $\Horn_4^{\nr}\cap {\overline\Horn}_4$ has 12
  extremal rays generated by
  the triples $(\lambda,\mu,\nu)$
  associated to the following inclusions
  \begin{enumerate}
  \item $V(1)\subset V(1)\otimes V(0)$, $V(1)\subset V(111)\otimes
    V(0)$, $V(1)\subset V(0)\otimes V(1)$, $V(11)\subset V(0)\otimes
    V(11)$  and $V(111)\subset V(0)\otimes V(111)$ ;
\item $V(1^2)\subset V(1)\otimes V(1)$ ;
\item $V(1^4)$ is contained in $V(1)\otimes V(1^3)$ and $V(1^3)\otimes
  V(1)$ ;

\item $V(2211)$ is contained in $V(1^3)\otimes V(1^3)$ and $V(211)\otimes
  V(11)$ ;
\item $V(1^3)\subset V(1)\otimes V(11)$ ;
\item $V(21^3)\subset V(1^3)\otimes V(1^2)$.
  \end{enumerate}
Each triple $(\lambda,\mu,\nu)$ on one of these  extremal rays indexes a Littlewood-Richardson coefficient with value 1.  
The Hilbert basis of $\Horn_4^{\nr}\cap {\overline\Horn}_4$ consists in these 12 generators.
\end{prop}

\subsection{Special case of self-dual representations}
\label{sec:special_case_m-0-0}

Let $k$ and $l$ be two nonnegative integers and $n\geq 4$. 
The $\SL_n(\CC)$-modules  $V_n((2k)k^{n-2})$,  $V_n((2l)l^{n-2})$ and
hence  $V_n((2k)k^{n-2}) \otimes V_n((2l)l^{n-2})$ are self-dual.

In \cite[Section~8]{PW:multiplicities}, conjectural values (for $n=6$)
 are given for the numbers of isotypic components in
 $V_n((2k)k^{n-2}) \otimes V_n((2l)l^{n-2})$ and for the numbers of self-dual isotypic components. 
Here, we prove and extend these formulas.

\begin{coro}
\label{cor:nbisotnm}
Assume up to symmetry that $l\leq k$.
  The number of distinct isotypic components in $V_n((2k)k^{n-2})
  \otimes V_n((2l)l^{n-2})$ is given by
$$
\left\{
  \begin{array}{l@{\ \mathrm{if}\ }l}
    l^3+3l^2+3l+1& 2l\leq k\\
\frac{1}{3}k^{3} - 2 k^{2} l + 4 k l^{2} - \frac{5}{3} l^{3} -  k^{2}
    + 4kl -  l^{2} + \frac{2}{3} k+ \frac{5}{3} l + 1
&2l\geq k
  \end{array}
\right .
$$
\end{coro}

\begin{proof}
By Proposition~\ref{prop:stab}, one may assume that $n=4$.
Then, Proposition~\ref{prop:Horn4nr2} shows that $\nu\in\ZZ^4$ is the
highest weight of an isotypic component of   $V_4((2k)k^{2}) \otimes
V_4((2l)l^{2})$
if and only if (recall that $l\leq k$) all of the following conditions
hold:

\begin{equation}
\begin{array}{c@{\ \quad\ }c}
  4(k+l)=|\nu|& \nu_1\geq \nu_2\\
\nu_4\geq 0&\nu_3+\nu_4\geq k+l\\
\multicolumn{2}{c}{\nu_1+\nu_3\geq 3k+l}\\
\multicolumn{2}{c}{2n+m\geq\nu_2\geq k+l\geq \nu_3\geq k}
\end{array}
\label{eq:in6}
\end{equation}
The corollary follows by explicit computations that can be performed with \cite{iscc}.
\end{proof}

Similarly, one gets the number of self-dual representations.

\begin{coro}
\label{cor:nbisotnm2}
Assume up to symmetry that $l\leq k$.
The $\SL_n(\CC)$-module $V_n((2k)k^{n-2})
  \otimes V_n((2l)l^{n-2})$ contains $(l+1)^2$
  distinct self-dual isotypic components.
\end{coro}

\begin{proof}
By Proposition~\ref{prop:stab}, one may assume that $n=4$.
Then, the set of self-dual isotypic components of   $V_4((2k)k^{2}) \otimes
V_4((2l)l^{2})$ are obtained by adding the condition
$$
\nu_1+\nu_4=\nu_2+\nu_3
$$
to the conditions~\eqref{eq:in6}.
The corollary follows by explicit computations that can be performed with \cite{iscc}.
\end{proof}

\subsection{Computation of $\Nb(c_{\lambda\mu}^\bullet> c)$ for $\lambda$ and $\mu$ near-rectangular}
\label{sec:SL4nr2}

In this subsection, we report on the computation of the function
$$
\begin{array}{cccl}
 \Nb_4(c_{\lambda\mu}^\bullet> c)\,:
  &(\Lambda_4^\nr)^2\times\NN&\longto&\NN\\
&(\lambda,\mu,c)&\longmapsto&\#\{
\nu\in\Lambda_4\,:\,c_{\lambda\mu}^\nu> c\}.
\end{array}
$$
By Proposition~\ref{prop:stab}, this function determines
$\Nb_n(c_{\lambda\mu}^\bullet> c)$ for any near-rectangular
partitions $\lambda$ and $\mu$ of length $n\geq 4$.

Since Propositions~\ref{prop:LR3} and \ref{prop:stab} give similar
expressions for the Littlewood-Richardson coefficient, we can apply
the strategy of Section~\ref{sec:SL3}. 

We get that $\Nb_4(c_{\lambda\mu}^\bullet> c)$ is the number of points $\nu\in\Lambda_4$ such that $\lambda_1+2\lambda_2+\mu_1+2\mu_2=\nu_1+\nu_2+\nu_3+\nu_4$ and
 
\[
\begin{array}{l@{\qquad}l}
-\lambda_2-\mu_2+\nu_2\geq 0 & \lambda_2+\mu_2-\nu_3\geq 0\\
\lambda_1-\lambda_2\geq c & -\lambda_2+\nu_3\geq c\\
\nu_1-\nu_2\geq c & \nu_3-\nu_4\geq c\\
\lambda_1+\mu_1-\nu_1\geq c & \lambda_2+\mu_1-\nu_2\geq c\\
-\lambda_2-\mu_2+\nu_3+\nu_4\geq c & -\mu_2+\nu_3\geq c\\
\lambda_1+\mu_2-\nu_2\geq c & \lambda_1+\lambda_2+\mu_2-\nu_2-\nu_4\geq c\\
\lambda_1+\mu_1+\mu_2-\nu_1-\nu_4\geq c & \lambda_2+\mu_1+\mu_2-\nu_2-\nu_4\geq c
\end{array}
\]

In particular, it is the number of integer points in some polytope depending linearly on the data $(\lambda,\mu,c)$. 
Then $\Nb_4(c_{\lambda\mu}^\bullet> c)$ is piecewise quasi-polynomial, and
 can be computed using Barvinok's algorithm. Surprisingly, here
 $\Lambda=(\Lambda_4^\nr)^2\times\ZZ$ and $\Nb_4(c_{\lambda\mu}^\bullet> c)$ is piecewise polynomial.

As in Section~\ref{sec:SL3}, from this point on we use the basis of
fundamental weights to write $\lambda=k_1\varpi_1+k_2\varpi^*_1$ and
$\mu=l_1\varpi_1+l_2\varpi^*_1$. Thus the symmetry we want to observe
is once again with respect to swapping  $k_1$ and $k_2$. Consider the function
\[ \begin{array}{rccl}
\psi: & \NN^5 & \longrightarrow & \NN\\
 & (k_1,k_2,l_1,l_2,c) & \longmapsto & \#\{\nu\in\Lambda_4\,|\,c_{k_1\varpi_1+k_2\varpi^*_1,l_1\varpi_1+l_2\varpi^*_1}^\nu >c\}
\end{array}. \]

We now give details about the results in
Proposition~\ref{prop:introSL4nr2}:

\begin{prop}
\label{prop:calculSL4nr2}
  We have $\psi(k_1,k_2,l_1,l_2,c)=0$ unless
\[
c\leq\min(k_1,k_2,l_1,l_2).
\]
Moreover, this cone decomposes into 36 cones $C_1,\dots,C_{36}$ on which $\psi$ is given by 
polynomial functions $P_1,\dots,P_{36}$. 12 of these pairs
$(C_i,P_i)$ are kept unchanged by swapping $k_1$ and $k_2$ (namely $(C_1,P_1)$ to $(C_{12},P_{12})$). The 24 other such
pairs are pairwise swapped by this operation (for all $i\in\{7,\dots,18\}$, $(C_{2i-1},P_{2i-1})$ and $(C_{2i},P_{2i})$ are swapped).

In particular, Conjecture~\ref{conj:setLR} holds for $\GL_4$ and $\lambda,\mu$ near-rectangular. 
\end{prop}

Now to present as clearly as possible these cones and polynomials without writing all of them, let us use the two following involutions: $s_1$ corresponding to the exchange of $k_1$ and $k_2$, and $s_2$ corresponding to swapping $(k_1,k_2)$ and $(l_1,l_2)$. Then $\langle s_1,s_2\rangle$ acts on the set of all pairs $(C_i,P_i)$ with 8 orbits. Let us give below one representative for each one of these. The labelling is the one of the complete list \cite[pol\_and\_cones\_SL4nr2.txt]{wp}, chosen so that the stability when swapping $k_1$ and $k_2$ is easier to see:
\[ C_1:\quad l_1+l_2\leq k_1+c,\quad l_1+l_2\leq k_2+c, \]
\[ P_1=\left(-\frac{1}{2}\right) \cdot (-l_{2} + c - 1) \cdot (-l_{1} + c - 1) \cdot (-l_{1} - l_{2} + 2 c - 2) \]
has a $\langle s_1,s_2\rangle$-orbit of size 2;
\[ C_{16}:\quad l_1+l_2\leq k_1+c,\quad l_1+l_2\geq k_2+c,\quad k_2\geq l_1,\quad k_2\geq l_2, \]
\[ P_{16}=P_1-\begin{pmatrix}
-k_2+l_1+l_2-c+2\\
3
\end{pmatrix} \]
has an orbit of size 4;
\[ C_2:\quad l_1+l_2\geq k_1+c,\quad l_1+l_2\geq k_2+c,\quad k_1\geq l_1,\quad k_1\geq l_2,\quad k_2\geq l_1,\quad k_2\geq l_2, \]
\[ P_2=P_{16}-\begin{pmatrix}
-k_1+l_1+l_2-c+2\\
3
\end{pmatrix} \]
has an orbit of size 2;
\[ C_{19}:\quad l_1+l_2\geq k_1+c,\quad k_1\geq l_1,\quad k_2\leq l_1,\quad k_2\geq l_2, \]
\[ P_{19}=P_2+\begin{pmatrix}
-k_2+l_1+1\\
3
\end{pmatrix} \]
has an orbit of size 8;
\[ C_{21}:\quad l_1+l_2\leq k_1+c,\quad k_2\leq l_1,\quad k_2\geq l_2, \]
\[ P_{21}=P_{16}+\begin{pmatrix}
-k_2+l_1+1\\
3
\end{pmatrix} \]
has an orbit of size 8;
\[ C_{29}:\quad k_1+k_2\geq l_1+l_2,\quad l_1+l_2\geq k_1+c,\quad k_2\leq l_1,\quad k_2\leq l_2, \]
\[ P_{29}=P_{19}+\begin{pmatrix}
-k_2+l_2+1\\
3
\end{pmatrix} \]
has an orbit of size 4;
\[ C_{27}:\quad k_1+k_2\leq l_1+l_2,\quad k_1\geq l_1,\quad k_1\geq l_2, \]
\[ P_{27}=P_{29}+\begin{pmatrix}
-k_1-k_2+l_1+l_2+1\\
3
\end{pmatrix} \]
has an orbit of size 4; finally,
\[ C_{36}:\quad l_1+l_2\leq k_1+c,\quad k_2\leq l_1,\quad k_2\leq l_2, \]
\[ P_{36}=P_{21} +
\begin{pmatrix}
-k_2+l_2+1\\
3
\end{pmatrix} \]
also has an orbit of size 4.

\bigskip
\begin{rem}
  One can observe that the polynomials $P_i$ are expressed using
  each other. 
We exploit here the fact that the difference between two polynomials
associated to two adjacent cones has a simple expression theoretically
given by the Paradan formula \cite{Paradan:jump,BV:Paradan}. 
\end{rem}
 
\subsection{Computation of $\Nb_4(c_{\lambda\mu}^\bullet> 0)$ for $\lambda$  near-rectangular}\label{sec:GL4conj2}

In this section, we report on the computation of the function
$$
\begin{array}{cccl}
 \Nb_4(c_{\lambda\mu}^\bullet> 0)\,:
  &\Lambda_4^\nr\times\Lambda_4&\longto&\NN\\
&(\lambda,\mu)&\longmapsto&\#\{
\nu\in\Lambda_4\,:\,c_{\lambda\mu}^\nu> 0\}.
\end{array}
$$
As we recalled in Proposition~\ref{prop:nbisquasipol},
$\Nb_4(c_{\lambda\mu}^\bullet>0)$ is the number of integer points
in an affine section of the Horn cone. The inequalities defining this cone
are explicitly given in Proposition~\ref{prop:horn4nr}.
Then, one can compute explicitly the quasi-polynomial function with
the program \cite{iscc}. The output is too big (even using symmetries)
to be collected there. The interested reader can get details from
\cite[Supplementary material]{wp}.

\begin{prop}
\label{prop:calculSL4nr}
The cone $\Lambda_4^\nr\times\Lambda_4$ decomposes into 205 cones of non
empty interior. On 151 of them $\Nb_4(c_{\lambda\mu}^\bullet> 0)$ is
polynomial of degree 3, and on the other 54 it is quasi-polynomial. 
The only congruence occurring is the parity of $\lambda_1+|\mu|$.
  
Moreover, for any pair $(C,P)$ where $C$ is one of the 205 cones and $P$ the corresponding
function, one can see that in this list there is also a pair $(C',P')$ obtained by replacing $\lambda$
by $\lambda^*$ (in 57 cases, $(C',P')=(C,P)$). In particular,
Conjecture~\ref{conj:setLR} holds. 

Under the action of $\ZZ/2\ZZ\times \ZZ/2\ZZ$ there are $61$ orbits of
actual polynomials and $22$ orbits of quasi-polynomials. 
\end{prop}

Here we give three examples illustrating some of the variety of cases that one can observe. The
function $\Nb_4(c_{\lambda\mu}^\bullet> 0)$ for
$\lambda=k_1\varpi_1+k_2\varpi^*_1\in\Lambda_4^\nr\cap\Lambda_4^0$ and $\mu=\mu_1\mu_2\mu_3\in\Lambda_4^0$ is given:
\begin{itemize}
\item on the cone defined by $\mu_1\geq k_1 + \mu_3$, $\mu_1\geq k_2 + \mu_3$, $\mu_2\geq \mu_3$, $k_1 + k_2 + \mu_3 \geq \mu_1 + \mu_2$, $\mu_3\geq 0$, by the polynomial
$$
\begin{array}{l}
P=\displaystyle\frac{\mu_3}{2} \cdot \bigg( \mu_2(2\mu_1-\mu_2+1) + 2(\mu_1+1) - (\mu_3+1)(k_1+k_2+\mu_1-\mu_2+2) \bigg)\\
-\displaystyle\frac{\mu_2+1}{6}\cdot \bigg (3(k_1^{2} + k_2^{2}) - 3(k_1+k_2) (2 \mu_1+1)  + 3 \mu_1^{2} + 2 \mu_2^{2} - 3 \mu_1 + 4 \mu_2 - 6\bigg),
\end{array}
$$
symmetric in $k_1,k_2$.
\item on the cone defined by $\mu_1 + \mu_2 \geq k_1 + k_2 + \mu_3$, $k_2 + \mu_1 \geq k_1 + \mu_2 + \mu_3$, $k_2 + \mu_3 \geq \mu_2$, $k_1 + \mu_1 \geq k_2 + \mu_2 + \mu_3$, $k_1 + \mu_3 \geq \mu_2$, $k_1 + k_2 \geq \mu_1$, $\mu_3\geq 0$ (adjacent to the previous one), by the quasi-polynomial
\[ \left\lbrace\begin{array}{ll}
P+\displaystyle\frac{1}{24} \, {\left(k_{1} + k_{2} - \mu_{1} - \mu_{2} + \mu_{3} - 1\right)} & \\
\cdot{\left(k_{1} + k_{2} - \mu_{1} - \mu_{2} + \mu_{3} + 1\right)} & \text{if }k_1+k_2+\mu_1+\mu_2+\mu_3\text{ is odd}\\
\cdot{\left(-2 \, k_{1} - 2 \, k_{2} + 2 \, \mu_{1} + 2 \, \mu_{2} + 4 \, \mu_{3} + 3\right)} & \\
 & \\
P+\displaystyle\frac{1}{24} \, {\left(k_{1} + k_{2} - \mu_{1} - \mu_{2} + \mu_{3}\right)} & \\
\cdot\Big(2+{\left(k_{1} + k_{2} - \mu_{1} - \mu_{2} + \mu_{3}\right)} & \text{if }k_1+k_2+\mu_1+\mu_2+\mu_3\text{ is even},\\
\cdot{\left(-2 \, k_{1} - 2 \, k_{2} + 2 \, \mu_{1} + 2 \, \mu_{2} + 4 \, \mu_{3} + 3\right)}\Big) & 
\end{array}\right. \]
also symmetric in $k_1,k_2$.
\item on the cone defined by $\mu_1 \geq k_1$, $\mu_1 \geq k_2 + \mu_3$, $\mu_2 \geq \mu_3$, $k_2 \geq \mu_2$, $k_1 + \mu_3 \geq \mu_1$ (also adjacent to the first one), by the non-symmetric polynomial
\[ P+\begin{pmatrix}
k_1-\mu_1+\mu_3+1\\
3
\end{pmatrix}. \]
\end{itemize}

\section{Related questions}

\subsection{In type $D_n$}

Apart from type $A_n$, it is only in types $D_n$ and $E_6$ that the irreducible representations of simple Lie algebras are not all self-dual. Consider here type $D_5$.

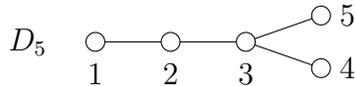
\begin{center}
  \begin{tikzpicture}[scale=0.5]
    \draw (-1,0) node[anchor=east] {$D_{5}$}; \draw (0 cm,0) -- (4
    cm,0); \draw (4 cm,0) -- (6 cm,0.7 cm); \draw (4 cm,0) -- (6
    cm,-0.7 cm); \draw[fill=white] (0 cm, 0 cm) circle (.25cm)
    node[below=4pt]{$1$}; \draw[fill=white] (2 cm, 0 cm) circle
    (.25cm) node[below=4pt]{$2$}; \draw[fill=white] (4 cm, 0 cm)
    circle (.25cm) node[below=4pt]{$3$}; \draw[fill=white] (6 cm, 0.7
    cm) circle (.25cm) node[right=3pt]{$5$}; \draw[fill=white] (6 cm,
    -0.7 cm) circle (.25cm) node[right=3pt]{$4$};
  \end{tikzpicture}
\end{center}

Let $(\varpi_1,\dots,\varpi_5)$ be the list of fundamental
weights. Then $V({\varpi_4})^*\simeq V({\varpi_5})$ whereas
$V(\varpi_1)$, $V(\varpi_2)$ and $V(\varpi_3)$ are self-dual. 
The natural generalization of near-rectangular partitions is then to
consider the dominant
weights in $\NN \varpi_4\oplus\NN \varpi_5$.
A natural generalization of Conjecture~\ref{conj:nbLR} would be:  for $\lambda=a \varpi_4+b\varpi_5\in \NN \varpi_4\oplus\NN \varpi_5$ and $\mu$ a
dominant weight of $D_5$,  do the two tensor products 
$$V_{D_5}(a \varpi_4+b\varpi_5)\otimes V_{D_5}(\mu)\qquad{\rm and}
\qquad
V_{D_5}(b \varpi_4+a\varpi_5)\otimes V_{D_5}(\mu)$$
contain the same number of isotypic components?

The answer is NO, even assuming that $\mu\in \NN \varpi_4\oplus\NN
\varpi_5$ too.
An example is $\lambda=2\varpi_4+\varpi_5$ and
$\mu=\varpi_4+2\varpi_5$.
The two tensor products have respectively $31$ and $30$ isotypic
components as checked using SageMath \cite{sage}:\\

\noindent
{\tt
sage: D5=WeylCharacterRing("D5",style="coroots")\\
sage: len(D5(0,0,0,2,1)*D5(0,0,0,1,2))\\
31\\
sage: len(D5(0,0,0,1,2)*D5(0,0,0,1,2))\\
30\\
}

\subsection{In type $A_n$}
\label{sec:typeAn}

The representations of $\SL_n(\CC)$ corresponding to near-rectangular
partitions are of the form $V(a\varpi_1+b\varpi_{n-1})$. 
Observe that $(\varpi_1,\varpi_{n-1})$ is a pair of mutually  dual
fundamental weights. 
One could hope that Conjecture~\ref{conj:setLR}  or \ref{conj:nbLR}  hold for any linear
combination of a given pair of mutually dual fundamental weights. This
is not true even for $(\varpi_2,\varpi_3)$ and $n=5$. Indeed, for
$\lambda=\varpi_2+2\varpi_ 3$ and $\mu=3\varpi_2+\varpi_3$, the numbers
of isotypic components in $V(\lambda)\otimes V(\mu)$ and
$V(\lambda)^*\otimes V(\mu)$ differ:\\

\noindent
{\tt
sage: len(lrcalc.mult([3,3,2],[4,4,1],5))\\
34\\
sage: len(lrcalc.mult([3,3,1],[4,4,1],5))\\
33\\
}

Mention finally that we checked Conjecture~\ref{conj:setLR} on
examples, using  SageMath. See \cite[test\_Conj1.sage]{wp}:

\begin{itemize}
\item Conjecture~1 holds for $\GL_4$
 if $\max(\lambda_1-\lambda_2,\lambda_2)\leq 20$ and
  $|\mu|\leq 40$.
\item Conjecture~1 holds for $\GL_5$
if $\max(\lambda_1-\lambda_2,\lambda_2)\leq 20$ and
  $|\mu|\leq 30$.
\item Conjecture~1 holds for $\GL_6$
if $\max(\lambda_1-\lambda_2,\lambda_2)\leq 10$ and
  $|\mu|\leq 30$.
\item Conjecture~1 holds for $\GL_{10}$
if $\max(\lambda_1-\lambda_2,\lambda_2)\leq 10$ and
  $|\mu|\leq 15$.
\end{itemize}

\bibliographystyle{alpha}
\bibliography{lr_PR}


\end{document}